\documentclass{amsart}
\usepackage{mathrsfs}

\newtheorem{theorem}{Theorem}

\theoremstyle{definition}

\theoremstyle{remark}

\numberwithin{equation}{section}

\begin{document}

\title[One-sentence elementary proof of Fa\`a di Bruno's formula]{A one-sentence elementary proof of the combinatorial Fa\`a di Bruno's formula}

\author{Samuel Allen Alexander}

\maketitle

Fa\`a di Bruno's formula generalizes the chain rule of elementary calculus.
It tells us the $n$th derivative $f(g(x))^{(n)}$ of a composition of functions.
See
\cite{johnson2002curious} and
\cite{craik2005prehistory}
for the remarkable history of this formula.
In the past two decades, shockingly simple proofs were found
\cite{hardy2006combinatorics}
\cite{ma2009higher}
for multivariable generalizations of the formula.
The multivariable nature of these proofs, however, makes them
less accessible than they should, in our opinion, be.
We present the single-variable special-case, which is so simple
that it can be expressed in a single sentence.

\begin{theorem}
\label{mainthm}
    (Combinatorial Fa\`a di Bruno's formula)
    Let $n>0$.
    If $f$ and $g$ possess derivatives up to order $n$, then
    \[
        f(g(x))^{(n)} = \sum_{\pi\in\Pi_n}f^{(|\pi|)}(g(x))
            \prod_{B\in\pi} g^{(|B|)}(x)
    \]
    where $\pi$ ranges over the set $\Pi_n$ of all partitions of $\{1,\ldots,n\}$
    and $B$ ranges over the blocks in $\pi$.
\end{theorem}

\begin{proof}
    The theorem is proved by induction on $n$, using the fact that,
    for each partition $\pi=\{B_1,\ldots,B_k\}$ of $\{1,\ldots,n\}$,
    rewriting $f^{(|\pi|)}(g(x))\prod_{B\in\pi} g^{(|B|)}(x)$
    as $f^{(k)}(g(x))g^{(|B_1|)}(x)\cdots g^{(|B_k|)}(x)$,
    the product rule says that the derivative thereof
    equals a sum of one term in which $f^{(k)}(g(x))$ is replaced by
    $f^{(k+1)}(g(x))g^{(1)}(x)$ (this term corresponds to the partition of $\{1,\ldots,n+1\}$
    obtained from $\pi$ by adding $\{n+1\}$ as a new block)
    plus $k$ terms in which a factor $g^{(|B_i|)}(x)$ is replaced by
    $g^{(|B_i|+1)}(x)$ (which corresponds to the partition of $\{1,\ldots,n+1\}$ obtained
    from $\pi$ by adding $n+1$ to $B_i$), and every partition of $\{1,\ldots,n+1\}$
    is obtained from exactly one partition of $\{1,\ldots,n\}$ in exactly one of these ways.
\end{proof}

The non-combinatorial version of Fa\`a di Bruno's formula,
\[
    f(g(x))^{(n)}
    =
    \sum_{k_1+2k_2+\cdots+nk_n=n}
    \frac{n!}{k_1!1!^{k_1}\cdots k_n!n!^{k_n}}
    f^{(k_1+\cdots+k_n)}(g(x))\prod_{i=1}^n g^{(i)}(x)^{k_i},
\]
follows by using a simple counting argument to show that the number of
partitions of $\{1,\ldots,n\}$ with $k_i$ size-$i$ blocks ($1\leq i\leq n$)
is $n!/k_1!1!^{k_1}\cdots k_n!n!^{k_n}$.

\bibliographystyle{plain}
\bibliography{main}

\begin{thebibliography}{1}

\bibitem{craik2005prehistory}
Alex Craik.
\newblock Prehistory of {F}a{\`a} di {B}runo's formula.
\newblock {\em The American Mathematical Monthly}, 112(2):119--130, 2005.

\bibitem{hardy2006combinatorics}
Michael Hardy.
\newblock Combinatorics of partial derivatives.
\newblock {\em The {E}lectronic {J}ournal of {C}ombinatorics}, 2006.

\bibitem{johnson2002curious}
Warren~P Johnson.
\newblock The curious history of {F}a{\`a} di {B}runo's formula.
\newblock {\em The American Mathematical Monthly}, 109(3):217--234, 2002.

\bibitem{ma2009higher}
Tsoy-Wo Ma.
\newblock Higher chain formula proved by combinatorics.
\newblock {\em The {E}lectronic {J}ournal of {C}ombinatorics}, 2009.

\end{thebibliography}

\end{document}